\documentclass[12pt]{amsart}

\usepackage{amsmath,amsthm,amsfonts,amssymb,amscd}
\usepackage[all]{xy}
\usepackage{sseq}

\theoremstyle{plain}
\newtheorem{thm}{Theorem}[section]
\newtheorem{prop}[thm]{Proposition}

\newtheorem{cor}[thm]{Corollary}

\theoremstyle{definition}

\theoremstyle{remark}

\numberwithin{equation}{section}

\newcommand{\Z}{\mathbb{Z}}

\newcommand{\Q}{\mathbb{Q}}
\newcommand{\R}{\mathbb{R}}

\newcommand{\C}{\mathbb{C}}

\renewcommand{\phi}{\varphi}

\DeclareMathSymbol{\sdp}{\mathbin}{AMSb}{"6F}

\newcommand{\ignore}[1]{}

\begin{document}
\baselineskip=16pt

\begin{abstract}
We produce a formula for the $\Z_2$-Betti numbers of the moduli space  $M_r^d$ of stable real Higgs bundles over a real projective curve, with coprime rank $r$ and degree $d$.  Our approach relies on the motivic formula for the moduli space due to Mellit \cite{M}, Fedorov-Soibelman-Soibelman \cite{FSS}, and Schiffmann \cite{S}, and the fact that the virtual $\Z_2$ Poincar\'e polynomial is a motivic measure over $\R$.
\end{abstract}

\title{Betti numbers of the moduli space of Higgs bundles over a real curve}
\address{ 	Department of Mathematics and Statistics,
Memorial University of Newfoundland,
St. John's, NL,
Canada,
A1C 5S7,
tbaird@mun.ca}
\author{Thomas John Baird}
\keywords{Moduli spaces of Higgs bundles, symmetric products of a curve, real algebraic geometry, motivic ring, Poincar\'e polynomials, Plethystic exponential}

\maketitle

\section{Introduction}

Let $\Sigma/k$ be a smooth, projective, geometrically connected curve of genus $g$ over a field $k$ admitting a $k$-rational point, and denote by $M_r^d$ the moduli space of stable Higgs bundles of rank $r$ and degree $d$ over $\Sigma$. When $gcd(d,r)=1$, which we assume henceforth,  $M_r^d$ is a smooth, quasi-projective variety.  

For $g,r,d \in \Z$ and $g, r \geq 0$, Schiffmann \cite{S} defines a Laurent polynomial
$$  A_{g,r,d} (q, \alpha_1,..., \alpha_g)  \in \Z[q, \alpha_1^{\pm 1},..., \alpha_g^{\pm 1} ]$$
which is symmetric under permutations of the $\alpha_i$ and under $\alpha_i \mapsto q \alpha_i^{-1}$. If $k = \mathbb{F}_q$ and $\{ \alpha_i, \alpha_i^{-1}q|i=1,..,g\}$ are set to the eigenvalues of the Frobenius action on $H^1_{et}(\Sigma, \Q_l)$ for $(q,l)=1$, then $A_{g,r,d}$ counts indecomposable bundles of rank $r$ and degree $d$ over $\Sigma$,  and  $q^{1+(g-1)r^2}A_{g,r,d}$ counts points in $M_r^d(\mathbb{F}_q)$. Using the Artin-Grothendieck Comparision Theorem, it follows that that when $k = \C$,  the Poincar\'e polynomial for the compactly supported rational Betti numbers of $M_r^d(\C)$ equals

\begin{equation}\label{Schiffresult}
   \sum_k  (-1)^kt^k \dim_\Q  H^k_{cpt}(M_r^d(\C);\Q) = t^{2(g-1)r^2+2}A_{g,r,d} (t^2, t,t,t,...,t,t).
   \end{equation}

Our main result is that a similar formula calculates $\Z_2$ Betti numbers of $M_r^d(\R)$ when $k =\R$. For a variety $X/\R$, denote its compactly supported Poincar\'e polynomial by
 $$P_t( X) :=   \sum_k  (-1)^k t^k \dim_{\Z_2} H^k_{cpt}(M_r^d(\R); \Z_2).$$

\begin{thm}\label{mainthm}
Suppose that $\Sigma$ is a smooth, projective, geometrically connected curve of genus $g$ defined over $k=\R$, and suppose that $\Sigma(\R)$ is a union of $b+1$ circles, with $b\geq 0$. Then 
$$P_t(M_r^d)  =  (-t)^{(g-1)r^2+1} A_{g,r,d} (-t, -t^{\frac{1}{2}}, -t^{\frac{1}{2}},  ..., -t^{\frac{1}{2}}, -1,-1,..,-1) $$ 
where  $(g-b)$-many entries equal $-t^{\frac{1}{2}}$ and $b$-many equal $-1$. 
\end{thm}

For example, for $r=2$, we obtain the explicit formula $$P_{t}(M_2^1) =  2^b t^{4(g-1)+1}(1-t)^g f(t^{1/2},1)$$ where

\begin{equation*}
\begin{split}
f(t^{\frac{1}{2}},z) &=  \frac{(z^2+t^{\frac{1}{2}})^{g-b} (z-t^{\frac{1}{2}})^{g-b} (z^2+1)^b(z-t)^b}{(z^2-1)(z+t)} + \frac{(z-t^{\frac{3}{2}})^{g-b} (1+t^{\frac{3}{2}})^{g-b} (z-t)^b(1+t^2)^b}{(z+t)(1-t^2)}\\
  &- \frac{(z+t^{\frac{1}{2}})^{g-b} (1-t^{\frac{1}{2}})^{g-b} (z+1)^b(1-t)^b}{2(z-1)(1+t)} - \frac{(z-t^{\frac{1}{2}})^{g-b} (1+t^{\frac{1}{2}})^{g-b} (z-1)^b(1+t)^b}{2(z+1)(1-t)}. 
\end{split}
\end{equation*}
A different formula for $P_t(M_2^1)$ was obtained in \cite{B}.  Formulas for $P_t(M_r^d)$ when $r\geq 3$ are new. 

We use a construction of $A_{g,r,d}$ due to Mellit \cite{M}.  Let $\mathcal{P}$ denote the set of Young diagrams.  Given $\mu \in \mathcal{P}$, define the rational function $\mathcal{H}_\mu$ in variables $ q, z, \alpha_1,..., \alpha_g$,

$$\mathcal{H}_\mu : =  \prod_{\square \in \mu}  \frac{\prod_{i=1}^g (z^{a(\square)+1} - \alpha_i q^{l(\square)})(z^{a(\square)} - \alpha_i^{-1} q^{l(\square)+1})}{(z^{a(\square)+1} - q^{l(\square)})(z^{a(\square)} - q^{l(\square)+1})}$$
where $a(\square)$ and $l(\square)$ denote the arm and leg lengths of boxes in the diagram.  For $g \geq 1$, define rational functions $H_{g,r}$ by
\begin{equation} \label{Plethlog}
 \sum_{r=1}^\infty H_{g,r} T^r = (z-1)(1-q) \operatorname{Log} \sum_{\mu \in \mathcal{P}} T^{|\mu|} \mathcal{H}_\mu, 
\end{equation}
where $\operatorname{Log}$ denotes the plethystic logarithm.  Mellit proved that for all $r \geq 1$, $H_{g,r}$ is a Laurent polynomial in $q$, $z$, and $\alpha_1, \ldots, \alpha_g$, and for all $d$,
\[
A_{g,r,d}(q, \alpha_1, \ldots, \alpha_g) = H_{g,r}(q, 1, \alpha_1, \ldots, \alpha_g).
\]

Note that $A_{g,r,d}$ is independent of $d$, so we denote it $A_{g,r}$ henceforth.   When applied to $k=\C$, Mellit's result proved a conjectural formula of Hausel and Rodriguez-Villegas for the $\Q$-Betti numbers of $M_r^d$.   Theorem \ref{mainthm} can be thought of as a Hausel Rodriguez-Villegas style formula for the $\Z_2$-Betti number of the moduli space of real Higgs bundles over a real curve.

Our proof of Theorem \ref{mainthm} can be outlined as follows.  Fedorov-Soibelman-Soibelman  \cite{FSS} proved a motivic version of Schiffmann's formula which, when $k$ is a field of characteristic zero, identifies the class $[M_r^d]$ in $\overline{Mot}(k)$, the dimensional completion of the Grothendieck ring of Artin stacks over $k$.  Mellit (\cite{M} \S 6), identifies their class $[M_r^d]$ as the image of $q^{(g-1)r^2+1} A_{r,d}$  under a homomorphism 

\begin{equation}\label{evsigma}
 ev_\Sigma  :  R_g  \rightarrow \overline{Mot}(k),
 \end{equation} where $R_g$ is the ring of Laurent polynomials in $q, \alpha_1^{\pm 1},...,\alpha_g^{\pm 1}$ which are invariant under permutations of $\alpha_i$ and under $\alpha_i \mapsto q \alpha_i^{-1}$.  This $ev_\Sigma$ sends $q$ to the Lefschetz class $\mathbb{L}$ and the extension to $R_g [[z]] \rightarrow \overline{Mot}(k)[[z]]$ sends the formal zeta function of the curve to the motivic zeta function of $\Sigma$ :
 
\begin{equation}\label{evfor}
 ev_\Sigma \left( \frac{ \prod_{i=1}^g (1-\alpha_i z)(1-\alpha_i^{-1}qz)}{(1-z)(1-qz)}\right) = \sum_{n=0}^{\infty} z^n  \Sigma^{(n)} = Z_\Sigma (z).
 \end{equation}

When $k= \R$, one can define a virtual Poincar\'e polynomial homomorphism $$P_t^{vir}:  \overline{Mot}(\R) \rightarrow \Z((t^{-1})) ,$$ which sends the class $[X]$ of a smooth, projective variety $X$ to the Poincar\'e polynomial  $  P_t^{vir}([X]) = P_t(X(\R))$.   

To calculate $P_t^{vir}( [M_r^d])$, it only remains  to determine the composition $P_t^{vir} \circ ev_\Sigma$, which boils down to calculating $P_t^{vir} ( Z_\Sigma(z))$. We carry this out in \S \ref{Betti numbers of}.  Then in \S \ref{On purity of real} we prove that the virtual Poincar\'e polynomial $M_r^d$ agrees with its topological Poincar\'e polynomial, completing the proof of Theorem \ref{mainthm}.  In section \ref{Average Betti numbers} we show that the Poincar\'e polynomial $P_t(M_r^d)$ is divisible by $P_t(Pic^0(\Sigma))$, which has the surprising implication that average value of the Poincar\'e polynomial of the connected components of $M_r^d$ is a polynomial with integer coefficients.

\section*{Acknowledgements}
Thank you to Oscar Garcia-Prada and to Ajneet Dhillon for suggesting that I use motivic formulas to calculate Betti numbers. Thanks also to Oscar Garcia-Prada and Jochen Heinloth who shared Maple code to calculate $[M^d_4]$ in terms of the motivic formula derived in \cite{GHS}.  Indeed the original plan for the current paper was to work from \cite{GHS},  and though their formula was later superseded by \cite{FSS}, it remained very helpful for confirming calculations.

\section{The Grothendieck Ring of Stacks}\label{The Grothendieck Ring}

We denote by $K_0(Var_k)$ the Grothendieck ring of varieties over the field $k$. This is the free abelian group generated by isomorphism classes $[X]$ of quasi-projective varieties over $X/k$ modulo the relation $[X] = [X\setminus Y] + [Y]$ where $Y\subseteq X$ is a closed subvariety. Multiplication is determined by the rule $[X] \cdot [Y] = [X\times_k Y]$.  Denote by $\mathbb{L}:= [\mathbb{A}^1] \in K_0(Var_k)$, called the Lefschetz element. 

Similarly, denote by $Mot(k)$ the Grothendieck ring of Artin stacks over $k$.  This is the free abelian group generated by isomorphism classes $[ \mathcal{X}]$ of algebraic stacks $\mathcal{X}/k$ of finite type with affine stabilizers,  modulo the relation $[\mathcal{X}] = [\mathcal{X} \setminus \mathcal{Y}] + [\mathcal{Y}]$ if $\mathcal{Y}$ is a closed substack  of $\mathcal{X}$ and the relation  $ [ \mathcal{E}] = [\mathcal{X}] [ \mathbb{A}^n]$ if $\mathcal{E}\rightarrow \mathcal{X}$ is a vector bundle of rank $n$ (this second relation is redundant for $K_0(Var_k)$). Multiplication is determined by $[\mathcal{X}] [\mathcal{Y}] = [\mathcal{X} \times \mathcal{Y}]$. 

There is a natural homomorphism $K_0(Var_k) \rightarrow Mot(k)$  sending $[X]$ to $[X]$. This descends to an isomorphism $ K_0'(Var_k) \cong Mot(k)$ where $K_0'(Var_k)$ is the localization of $K_0(Var_k)$ obtained by inverting $\mathbb{L}$ and $(\mathbb{L}^n-1)$ for all $n =1,2,3,...$ (\cite{E} Thm 1.2). 

We have a filtration $F^1 \subset F^2 \subset ... \subset  Mot(k)$ where $F^n$ is the subgroup generated by classes $[\mathcal{X}]$ with $\dim(\mathcal{X}) \leq -n$.  The dimensional completion $\overline{Mot}(k)$ is the completion of $Mot(k)$ under this filtration.

For a real projective variety $X/\R$, denote by $$P_t(X):=  \sum_{i=0}^{\infty} (-1)^i b_i(X) t^i$$  where $b_i(X) = \dim H^i_{cpt}(X(\R);\Z_2)$ is the dimension of the $i$th degree compactly supported cohomology with $\Z_2$ coefficients of $X(\R)$ endowed with analytic topology.  

\begin{prop}
There exists a unique motivic measure  $P_t^{vir}: K_0(Var_\R) \rightarrow \Z[t]$  such that for all smooth projective $X/\R$ we have $P_t^{vir}([X]) = P_t(X)$. This extends naturally to a ring homomorphism
$$P_t^{vir}: \overline{Mot}(\R) \rightarrow \Z[t] [[t^{-1}]] = \Z((t^{-1})).$$
\end{prop}

\begin{proof}
The homomorphism  $$P_t^{vir}: K_0(Var_\R) \rightarrow \Z[t] \subset \Z((t^{-1}))$$  is constructed by McCrory and Parusinski \cite{MP}.   Since $P_t^{vir}(\mathbb{L} ) = -t$   and $P_t^{vir}(\mathbb{L}^n -1)) = (-t)^n-1$ are both invertible in $\Z((t^{-1}))$, with $((-t)^n-1)^{-1} = (-t)^{-n} (\sum_{k=0}^{\infty} (-t)^{-kn})$, this extends to a homomorphism  $$P_t^{vir}: Mot(\R) \rightarrow   \Z((t^{-1})),$$  
which extends naturally to the dimensional completion.
\end{proof}

The zeta function of a variety $X$ is given by the formal power series
$$Z_X(z)  := \sum_{n=0}^\infty  [X^{(n)} ]z^n  \in  Mot(k) [[z]],  $$
where $X^{(n)}  = X^n/S_n $ denotes the $n$-fold symmetric product of $X$.  The zeta function is multiplicative in the sense that for closed $Y \subseteq X$,  $$ Z(X, z) = Z(Y,z) Z(X\setminus Y, z).$$

If $\Sigma/k$ is a smooth projective curve of genus $g$ and $\Sigma(k) \neq \emptyset$, then a theorem of Kapranov (\cite{K}  Theorem 1.1.9) says that

$$ Z_\Sigma(z) =  \frac{P(z)}{(1-z)(1-\mathbb{L}z)} $$ 
where $P(z)$ is a polynomial of degree $2g$ such that  $Z_\Sigma ( 1/\mathbb{L}z) = \mathbb{L}^{1-g} x^{2-2g} Z_\Sigma(z)$. The homomorphism (\ref{evsigma}) is determined by the property that 
$ \prod_{i=1}^g (1-\alpha_i z)(1-\alpha_i^{-1}qz)$ is sent to $P(z)$ in some algebraic extension of $\overline{Mot}(k)$

\section{Betti numbers of $\Sigma^{(n)}(\R)$}\label{Betti numbers of}

Suppose that $\Sigma$ is a smooth, geometrically connected, projective curve of genus $g$ over $\R$. For $n \geq 0$ denote by $\Sigma^{(n)}$ its $n$-th symmetric product. The set of complex points $\Sigma^{(n)}(\C)$ can be identified with the orbit space $$ Sym^n( \Sigma(\C)) := \Sigma(\C)^n/ S_n, $$  while $\Sigma^{(n)}(\R) \subset \Sigma^{(n)}(\C)$ can be identified with $Gal(\C/\R)$ fixed points. 

If $\Sigma(\R)$ is empty, then we have homeomorphisms
$$ \Sigma^{(n)}(\R) \cong  \begin{cases}  \emptyset  &  \text{if $n$ is odd} \\  Sym^{n/2}(N_g) & \text{if $n$ is even} \end{cases}$$
where $N_g$ is a non-orientable surface double covered by $\Sigma(\C)$.  We abusively write $ P_t(X)  :=  \sum_{k}  (-t)^k \dim H_{cpt}^k(X;\Z_2)$.  By Macdonald's formula \cite{Mac},
\begin{equation}\label{MacDs}
 \sum_{m=0}^\infty P_t( Sym^{m}(N_g)) z^m = \frac{ (1-tz)^{g+1} }{(1-z)(1-t^2z)},
 \end{equation}
 so

$$ P_t^{vir}(Z_\Sigma (z)) := \sum_{n=0}^{\infty} P_t ( \Sigma^{(n)} ) z^n =  \frac{ (1-tz^2)^{g+1} }{(1-z^2)(1-t^2z^2)}.$$

When $\Sigma(\R) \neq \emptyset$, $P_t(\Sigma^{(n)})$ was calculated in \cite{B}.  

\begin{prop}
If $\Sigma(\R)$ is non-empty, with $b+1$ many connected components, then
\begin{equation}\label{zetahom}
 P_t^{vir}(Z_\Sigma (z)) =   \sum_{n=0}^{\infty} P_t ( \Sigma^{(n)}) z^n   =  \frac{(1-tz^2)^{g-b} (1+z)^{b} (1-tz)^{b} }{(1-z)(1+tz)} .
 \end{equation}
\end{prop}

\begin{proof}

Suppose first that $ \Sigma(\C) \setminus \Sigma(\R)$ is connected.  For each pair of integers $k,s \geq 0$ such that $k+2s = n$, $\Sigma^{(n)}(\R)$ has $ { b+1 \choose k}$ connected components with Poincar\'e polynomial isomorphic to $(1-t)^k P_t( Sym^s(N_{g-k}))$.  Therefore, setting $a=b+1$, we have 

\begin{eqnarray*}
\sum_{n=0}^{\infty}  P_t(\Sigma^{(n)}) z^m  &=&  \sum_{0 \leq k \leq a} \sum_{s \geq 0}   z^{k+2s} { a \choose k}  (1-t)^k  P_t(  Sym^s(N_{g-k}))\\
&\stackrel{(\ref{MacDs})}{=}&  \sum_{0 \leq k \leq a}   z^{k} { a \choose k}   (1-t)^k \frac{(1-tz^2)^{g-k+1}}{(1-z^2)(1-t^2z^2)}\\
&=& \left(  (z -zt) + (1-tz^2) \right)^a  \frac{(1-tz^2)^{g-a+1}}{(1-z^2)(1-t^2z^2)} \\
&=&  \frac{(1-tz^2)^{g-a+1} (1+z)^{a-1} (1-tz)^{a-1}}{(1-z)(1+tz)} \\
\end{eqnarray*}

In case $\Sigma(\C) \setminus \Sigma(\R)$ is not connected, the calculation is the same except that whenever $k=a$, the corresponding component has Poincar\'e polynomial $(1-t)^{b+1}P_t(Sym^s(\Sigma_{(g-b)/2})) $, where $\Sigma_g$ is an orientable surface of genus $g$. But this equals $(1-t)^{b+1} P_t( Sym^s( N_{g-b-1}))$  (\cite{B} Remark 3.14) so the same formula holds.
\end{proof}

\begin{cor}
With hypotheses of Theorem \ref{mainthm} we have

$$ P_t^{vir}( [M_r^d]) =  (-t)^{(g-1)r^2+1} A_{g,r}(  -t, -t^{\frac{1}{2}},...,-t^{\frac{1}{2}},-1,...,-1).$$
\end{cor}

\begin{proof}
Substituting  $q = -t$,  $\alpha_1=...= \alpha_{g-b} = -t^{\frac{1}{2}}$ and $\alpha_{g-b+1}= ... =\alpha_g = -1$ into the formal zeta function $$\frac{ \prod_{i=1}^g (1-\alpha_i z)(1-\alpha_i^{-1}qz)}{(1-z)(1-qz)} $$
yields (\ref{zetahom}). Making these same substitutions into $A_{g,r}$ yields $P_t^{vir}( [M_r^d])$.
\end{proof}

\section{On purity of real semi-projective varieties}\label{On purity of real}

A semi-projective variety $Y$ over $\C$ is a smooth, quasiprojective variety equipped with $\C^\times$-action such that the fixed point set $Y^{\C^\times}$ is proper and for ever $y \in Y$ the limit $lim_{\lambda \rightarrow 0} \lambda y$ exists as $\lambda \in \C^\times$ tends to zero.  Mixed Hodge structures of such varieties were studied in Hausel and Rodriguez-Villegas \cite{HRV13} .

A semi-projective variety $X$ over $\R$ is a real variety equipped with an action of $\R^\times$ such that the base change  $X \times_R \C$ is semi-projective over $\C$.  When $\Sigma$ is defined over $\R$,  $M_r^d$ is a semi-projective variety, with respect the $\R^\times$ action by scaling the Higgs field.

\begin{prop}
If $X$ is semi-projective over $\R$ then
$$P_t^{vir} (X) = P_t (X).$$
\end{prop}

\begin{proof}

Under these conditions $ X^{\C^\times} = \coprod_{i \in I} F_i$  where the $F_i$ are non-singular, projective over $\R$, and we have a Bialycki-Birula stratification $ X = \bigcup_{i \in I} U_i$ where $$U_i(\C) :=  \{x \in X(\C) | lim_{\lambda \mapsto 0} \lambda x  \in F_i \} ,$$ and each $U_i$ is a locally closed subvarity of $X$, isomorphic to a vector bundle of some rank $d_i$ over $F_i$. It follows that

$$ P_t^{vir}(X) = \sum_{i\in I} P_t^{vir}(U_i) =   \sum_{i\in I} P_t^{vir}( F_i) P_t^{vir}(\mathbb{A}^{d_i}) =  \sum_{i\in I} P_t(F_i) (-t)^{d_i}.$$

On the other hand by a result essentially due to Duistermaat \cite{D}  (see also \cite{BGH}), we know that $ P_t( X) = \sum_i  P_t (U_i) $.  Since $U_i(\R)$ deformation retracts onto $F_i(\R)$ and has relative dimension $d_i$,  by Poincar\'e duality $P_t(U_i)=P_t(F_i) (-t)^{d_i}$ completing the proof.
\end{proof}

\section{Specializations}

\subsection{Maximality}

As a first application, we recover a result of Fu \cite{F}.  For a topological space $M$ and field $\mathbb{F}$, denote by $\beta_*(M;\mathbb{F}) = \sum_i \beta_i(M;\mathbb{F})$ the sum $\mathbb{F}$ of Betti numbers for $M$.  For a real variety $X$, the Smith-Thom inequality implies that  $\beta_*(X(\C);\Z_2 ) \geq \beta_*(X(\R); \Z_2)$, and $X$ is called maximal when equality holds. For example, a real curve $\Sigma$ of genus $g$ is maximal only if $\Sigma(\R)$ has $g+1$ components. 

In case, $X = M_r^d$,  $H^*(M_r^d(\C);\Z)$ is known to be torsion free by (\cite{GS}). Applying the universal coefficient theorem, $\beta_*( M_r^d(\C);\Q) = \beta_*(M_r^d(\C); \Z_2)$ so, by (\ref{Schiffresult})
$$ \beta_*( M_n^r(\C); \Z_2) = \mathcal{A}_{g,r}(1,-1,-1,...,-1). $$
On the other hand, applying Theorem \ref{mainthm} we deduce
$$ \beta_*(M_n^r(\R); \Z_2) = \mathcal{A}_{g,r}(1,\sqrt{-1},...,\sqrt{-1},-1,...,-1) $$
where $(g-b)$-entries equal $\sqrt{-1}$. If $\Sigma$ is maximal, then  $g-b =0$, so  $\beta_*( M_n^r(\C); \Z_2) = \beta_*(M_n^r(\R); \Z_2)$, so $M_n^r$ is maximal.

\subsection{Hodge polynomial specializations}

The Hodge structure of $M_n^d(\C)$ is pure, so Schiffmann's formula (\ref{Schiffresult}) can be refined to give the Hodge polynomial

$$ H_{u,v}(M_n^d) :=  \sum (-1)^{i+j}h^{i,j}(M_n^d) u^i v^j  =   \mathcal{A}_{g,r}(uv, u,u,...,u),$$
where $h^{i,j}(M_n^d) = \dim H^{(i,j)}_{cpt}(M_n^d(\C))$ are the (compactly supported) Hodge numbers.  If the real curve $\Sigma$ is maximal, then setting $u=t$ and $v = -1$ we obtain the identity
$$ H_{t,-1}(M_n^d)  =  P_t( M_n^d), $$
which means that the real variety $M_n^d$ is \emph{Hodge expressive} in the sense of Brugalle-Schaffhauser when $\Sigma$ is maximal.

At the other extreme,  if $\Sigma(\R)$ has only one path component so that $b=0$ in Theorem \ref{mainthm}, then we have

$$ H_{t^{\frac{1}{2}},-t^{\frac{1}{2}}} (M_n^d)  = P_t( M_n^d).  $$

\section{Average Betti numbers}\label{Average Betti numbers}

Our goal in this section is to prove the following somewhat surprising result.

\begin{prop}\label{avgbet}
Suppose that $\Sigma$ is a smooth projective, geometrically connected curve of genus $g$ defined over $\R$,  and $\Sigma(\R)$ has $b+1 >0$ connected components.
Then the Poincar\'e polynomial $P_t( M_r^d)$ is divisible by $2^b (1-t)^g$.
\end{prop}

\begin{proof}
Recall (\ref{Plethlog}):
\begin{equation}
 \sum_{r=1}^\infty H_{g,r} T^r = (z-1)(1-q) \operatorname{Log} \sum_{\mu \in \mathcal{P}} T^{|\mu|} \mathcal{H}_\mu.
\end{equation}

Every non-empty Young diagram $\mu$ includes at least one box with both arm and leg length zero, so $\mathcal{H}_\mu$ contains a factor of $ \prod_{i=1}^g (z - \alpha_i )(1 - \alpha_i^{-1} q)$ in its numerator. Using the combinatorics of the plethystic logarithm  (see for example  \cite{HRV08} Theorem 3.5.2.), it follows that every $H_{g,r}$ is divisible by $ \prod_{i=1}^g (z - \alpha_i )(1 - \alpha_i^{-1} q)$ in the localized ring $S^{-1} \Q[ q, z, \alpha_1^{\pm 1},...,\alpha_g^{\pm1}] $ where $S$ is the multiplicative set generated 
by  $\{  (z^{a+1} - q^{l} ),  (z^{a}-q^{l+1})|  a,b \geq 0\}$. We know
$$A_{g,r} (q, \alpha_1,..., \alpha_g) = H_{g,r}(q,1,\alpha_1,...,\alpha_g) \in \Z[q, \alpha_1^{\pm 1},..., \alpha_g^{\pm 1} ]$$ 
so by Gauss' Lemma $\prod_{i=1}^g (1-\alpha_i) (1-\alpha_i^{-1}q)$ divides $  A_{g,r}$ in $ \Z[q, \alpha_1^{\pm 1},..., \alpha_g^{\pm 1} ]$ .  Therefore by Theorem \ref{mainthm},   $P_t( M_r^d) = A_{g,r}(-t, -t^{\frac{1}{2}},...,-1,...) $  is divisible by   
$$2^b(1-t)^g = (1+t^{\frac{1}{2}})^{g-b}  (1-t^{\frac{1}{2}})^{g-b} (1+1)^b (1-t)^b $$ 
 in $\Z[t^{\pm \frac{1}{2}}]$, hence also in $ \Z[t]$. 
 \end{proof}

When the rank $r$ is odd, Proposition \ref{avgbet} can be proven using more elementary reasoning.  Choose a real line bundle  $L$ over $\Sigma$ of degree $d$, and let $M_r^L$ be the moduli space of stable Higgs bundles of rank $r$ with determinant $L$.  The morphism  $M_r^L \times Pic(\Sigma)^0 \rightarrow M_r^d$ sending the Higgs bundle $(E, \phi)$ to $(E \otimes L, \phi \otimes Id_L)$ descends to an isomorphism $ M_r^L \times_{\Gamma_r} Pic^0(\Sigma)  $ where $\Gamma_r \leq Pic^0(\Sigma)$ is the $r$-torsion subgroup scheme.

When $r$ is odd, this determines a homeomorphism  
\begin{equation}  \label{fixdetmap}
 M_r^L(\R) \times_{\Gamma_r(\R)}  Pic^0(\Sigma)(\R) \cong  M_r^d(\R). 
 \end{equation}
 
Here $M_r^L(\R)$ is connected, $Pic^0(\Sigma)(\R) \cong \Z_2^b \times (S^1)^g$ as a Lie group, and $\Gamma_r$ lies in the identity component of $Pic^0(\Sigma)(\R)$, which implies that $M_r^d(\R)$ has $2^b$ pairwise isomorphic connected components. Since $|\Gamma_r|$ and $|\Z_2|$ are relatively prime, 
\begin{align*}
 H^*( M_r^d(\R), \Z_2) & \cong  \left( H^*(  M_r^L(\R);\Z_2 ) \otimes  H^*(Pic^0(\Sigma)(\R);\Z_2) \right)^{\Gamma_r(\R)} \\ & \cong   H^*(  M_r^L(\R) ;\Z_2)^{\Gamma_r(\R)} \otimes  H^*(Pic^0(\Sigma)(\R);\Z_2),
 \end{align*}
so $P_t( M_r^d)$ must be divisible by $P_t( Pic^0(\Sigma)) =  2^b (1-t)^g$.

This argument fails when $r$ is even. In this case $\Gamma_r(\R)$ intersects every connected component of $ Pic^0(\R)$, so (\ref{fixdetmap}) maps to only one of the  $2^b$ components of $M_r^d(\R)$. Indeed, the connected components of $M_r^d(\R)$ have different Poincar\'e polynomials in general.  For example (see \cite{B} \S 5.2), when $g=b=2$, $M_2^1(\R)$ has four components, three of which have Poincar\'e polynomial   $t^5 (1-t)^2 (5t^3 - 7t^2 + 3t - 1)$    and one of which has Poincar\'e polynomial  $ t^5 (1-t)^2 (t^3 - 3t^2 + 3t - 1)$. Adding these together gives  $4 t^5 (1-t)^2( 4t^3-6t^2+3t-1) $. A priori, it is not at all obvious that the average value of the Poincar\'e polynomials of the connected components should have integer coefficients.

\ignore{

We remark that

$r=2$ $g=2$, $b=0,1,2$

$t^3 + t^2 + 2t + 2$

$2(t^2 + 2t^2 + 4t + 3)$

$4(t^3 + 3t^2 + 6t + 4)$

$r=2$ $g=3$, $b=0,1,2,3$

$t^6 + t^5 + 2t^4 + 5t^3 + 6t^2 + 6t + 3$

$2(t^6 + 2t^5 + 4t^4 + 8t^3 + 12t^2 + 12t + 5)$

$4(t^6 + 3t^5 + 7t^4 + 14t^3 + 22t^2 + 21t + 8)$

$6(t^6 + 4t^5 + 11t^4 + 24t^3 + 39t^2 + 36t + 13)$

$r=3$ $g=2$, $b=0,1,2$

$$ P_{-t}(M_r^d)/ (t+t^2)^2 =    $$

$t^8 + t^7 + 3t^6 + 6t^5 + 9t^4 + 16t^3 + 19t^2 + 16t + 6$

$2(t^8 + 2t^7 + 6t^6 + 13t^5 + 24t^4 + 41t^3 + 50t^2 + 39t + 13)$

$4(t^8 + 3t^7 + 10t^6 + 25t^5 + 54t^4 + 97t^3 + 121t^2 + 90t + 28)$

$r=3$, $g=3$, $b=0,1,2,3$

$$ P_{-t}(M_r^d)/ (t+t^2)^3 =   \left( \right)  $$

\begin{equation}
\begin{split}
\frac{P_{t}(M_2^1)}{2^b(t+t^2)^g} &= \Big( \frac{(z^2+t^{\frac{1}{2}})^{g-b} (z-t^{\frac{1}{2}})^{g-b} (z^2+1)^b(z-t)^b}{(z^2-1)(z+t)} + \frac{(z-t^{\frac{3}{2}})^{g-b} (1+t^{\frac{3}{2}})^{g-b} (z-t)^b(1+t^2)^b}{(z+t)(1-t^2)}\\
&  - \frac{(z+t^{\frac{1}{2}})^{g-b} (1-t^{\frac{1}{2}})^{g-b} (z+1)^b(1-t)^b}{2(z-1)(1+t)} - \frac{(z-t^{\frac{1}{2}})^{g-b} (1+t^{\frac{1}{2}})^{g-b} (z-1)^b(1+t)^b}{2(z+1)(1-t)}  \Big)|_{z=1}
\end{split}
\end{equation}

}


\begin{thebibliography}{9}

\bibitem[B]{B} Baird, Thomas John. ``Symmetric products of a real curve and the moduli space of Higgs bundles." Journal of Geometry and Physics 126 (2018): 7-21.


\bibitem[BGH]{BGH} Biss, Daniel, Victor W. Guillemin, and Tara S. Holm. ``The mod 2 cohomology of fixed point sets of anti-symplectic involutions." Advances in Mathematics 185.2 (2004): 370-399.


\bibitem[BS]{BS} Brugall\'e, Erwan, and Florent Schaffhauser. ``Maximality of moduli spaces of vector bundles on curves." \'Epijournal de G\'eom\'etrie Alg\'ebrique 6 (2023).

\bibitem[D]{D} Duistermaat, J. J. ``Convexity and tightness for restrictions of Hamiltonian functions to fixed point sets of an antisymplectic involution." Transactions of the American Mathematical Society 275.1 (1983): 417-429.


\bibitem[E]{E} Ekedahl, Torsten. ``The Grothendieck group of algebraic stacks." Perspectives on Four Decades of Algebraic Geometry, Volume 1: In Memory of Alberto Collino. Cham: Springer Nature Switzerland, 2024. 233-261.

\bibitem[FSS]{FSS} Fedorov, Roman, Alexander Soibelman, and Yan Soibelman. ``Motivic classes of moduli of Higgs bundles and moduli of bundles with connections." arXiv:1705.04890 (2017).

\bibitem[F]{F} Fu, Lie. ``Maximal real varieties from moduli constructions." Moduli 2 (2025): e8.

\bibitem[GHS]{GHS} Garcia-Prada, Oscar, Jochen Heinloth, and Alexander HW Schmitt. ``On the motives of moduli of chains and Higgs bundles." Journal of the European Mathematical Society 16.12 (2014): 2617-2668.

\bibitem[GS]{GS} Groechenig, Michael, and Shiyu Shen. "Complex $ K $-theory of moduli spaces of Higgs bundles." Journal of the European Mathematical Society (2025).


\bibitem[HRV08]{HRV08} Hausel, Tam\'as, and Fernando Rodriguez-Villegas. ``Mixed Hodge polynomials of character varieties: With an appendix by Nicholas M. Katz." Inventiones mathematicae 174.3 (2008): 555-624.

\bibitem[HRV13]{HRV13} Hausel, Tam\'as, and Fernando Rodriguez-Villegas. ``Cohomology of large semiprojective hyperk\"ahler varieties." arXiv:1309.4914 (2013).

\bibitem[K]{K} Kapranov, Mikhail. ``The elliptic curve in the S-duality theory and Eisenstein series for Kac-Moody groups." arXiv preprint math/0001005 (2000).


\bibitem[Mac]{Mac} Macdonald, Ian G. ``The Poincar\'e polynomial of a symmetric product." Mathematical Proceedings of the Cambridge Philosophical Society. Vol. 58. No. 4. Cambridge University Press, 1962.

\bibitem[MP]{MP} McCrory, Clint, and Adam Parusinski. ``Virtual Betti numbers of real algebraic varieties." Comptes Rendus. Mathematique 336.9 (2003): 763-768.

\bibitem[M]{M} Mellit, Anton. ``Poincar\'e polynomials of moduli spaces of Higgs bundles and character varieties (no punctures)." Inventiones mathematicae 221.1 (2020): 301-327.


\bibitem[S]{S}  Schiffmann, Olivier. ``Indecomposable vector bundles and stable Higgs bundles over smooth projective curves." Annals of Mathematics (2016): 297-362.

\end{thebibliography}
\end{document}